\documentclass[12pt]{article}

\usepackage{amsmath}
\allowdisplaybreaks[4]

\usepackage[all]{xy}
\usepackage{amssymb}
\usepackage{amsthm}
\usepackage{hyperref}
\hypersetup{colorlinks=true,linkcolor=blue,citecolor=red}
\usepackage{amsmath}
\usepackage{amscd}
\usepackage{verbatim}
\usepackage{eurosym}
\usepackage{float}
\usepackage{color}
\usepackage{dcolumn}
\usepackage[mathscr]{eucal}
\usepackage[all]{xy}
\usepackage{hyperref}
\usepackage{mathrsfs}
\usepackage{amsmath}
\usepackage{amssymb}
\usepackage{amsfonts,ifpdf}
\usepackage{graphicx}
\usepackage{times}
\usepackage{float}
\usepackage{epstopdf}
\usepackage{cite}
\usepackage{youngtab}
\usepackage{ytableau}
\ytableausetup
{mathmode, boxsize=0.9em}

\def\p{{\overline{p}}}
\def\N{{\overline{N}}}
\def\R{{\overline{R}}}

\usepackage{graphicx}

\makeatletter
\newcommand*\bigcdot{\mathpalette\bigcdot@{.6}}
\newcommand*\bigcdot@[2]{\mathbin{\vcenter{\hbox{\scalebox{#2}{$\m@th#1\bullet$}}}}}
\makeatother

\def\N{{\overline{N}}}

\newlength{\boxedparwidth}
  {\begin{center} \begin{tabular}{|@{\hspace{.315in}}c@{\hspace{.15in}}|}
                  \hline \\ \begin{minipage}[t]{\boxedparwidth}
                  \setlength{\parindent}{.25in}}%
  {\end{minipage} \\ \\ \hline \end{tabular} \end{center}}

\allowdisplaybreaks

\newtheorem{thm}{Theorem}[section]

\newtheorem{lem}[thm]{Lemma}

\makeatletter \@addtoreset{equation}{section} \makeatother
\makeindex 

\numberwithin{equation}{section}

\begin{document}

\baselineskip 20pt

\begin{center}
{\Large\bf Strict Log-Subadditivity for Overpartition Rank}\\ [7pt]
\end{center}

\vskip 3mm

\begin{center}
Helen W.J. Zhang$^{1,2}$ and Ying Zhong$^{1}$\\[8pt]
$^{1}$School of Mathematics\\
Hunan University\\
Changsha 410082, P. R. China\\[12pt]

$^{2}$Hunan Provincial Key Laboratory of \\
Intelligent Information Processing and Applied Mathematics\\
Changsha 410082, P. R. China\\[15pt]

Emails:  helenzhang@hnu.edu.cn,  YingZhong@hnu.edu.cn
\\[15pt]

\end{center}

\vskip 3mm
\begin{abstract}
Bessenrodt and Ono initially found the strict log-subadditivity of partition function $p(n)$, that is,
$p(a+b)< p(a)p(b)$ for $a,b>1$ and $a+b>9$.
Many other important statistics of partitions are proved to enjoy similar properties.
Lovejoy introduced the overpartition rank as an analog of Dyson's rank for partitions from the $q$-series perspective.
Let $\N(a,c,n)$ denote the number of overpartitions with rank congruent to $a$ modulo $c$.
Ciolan computed the asymptotic formula of $\N(a,c,n)$ and showed that $\N(a, c, n) > \N(b, c, n)$ for $c\geq7$ and $n$ large enough.
In this paper, we derive an upper bound and a lower bound of $\N(a,c,n)$ for each $c\geq3$ by using the asymptotics of Ciolan.
Consequently, we establish the strict log-subadditivity of $\N(a,c,n)$ analogous to the partition function $p(n)$.

\vskip 6pt

\\ [7pt]
\end{abstract}

\section{Introduction}
Strict log-subadditivity phenomena of partition statistics oriented from the work of Bessenrodt and Ono \cite{Bessenrodt-Ono-2016}.
Recall that a partition of a positive integer $n$ is a sequence of non-increasing positive integers whose sum equals $n$. Let $p(n)$ denote the number of partitions of $n$. Bessenrodt and Ono showed that $p(n)$ satisfies the strict log-subadditivity result, that is, for $a,b>1$ and $a+b>9$
\[p(a+b)< p(a)p(b).\]
Since then different partition functions are proven to enjoy the strict log-subadditivity as $p(n)$, for instance, the overpartition function \cite{Liu-Zhang-2021}, the
$k$-regular partition function \cite{Beckwith-Bessenrodt-2016} and the spt-function  \cite{Chen-2017, Dawsey-Masri-2019}.

Strict log-subadditivity of other statistics concerned with partitions comes into sight.
Dyson \cite{Dyson-1944} introduced the rank of a partition to interpret Ramanujan's congruences for $p(n)$ combinatorially.
This is defined as the largest part of the partition minus the number of parts.
Let $N(a,c,n)$  denote the number of partitions of $n$ with rank congruent to $a$ modulo $c$.
Hou and Jagadeesan \cite{Hou-Jagadeesan-2018} showed that
\begin{equation*}
N(r,3,a+b)<N(r,3,a)N(r,3,b)
\end{equation*}
when $r=0$ (resp. $r=1, 2$) and $a+b\geq12$ (resp. $11, 12$).
They also conjectured the general strict log-subadditivity: for $0\leq r<t$ and $t\geq 2$,
\begin{equation}\label{SN}
N(r,t,a+b)<N(r,t,a)N(r,t,b)
\end{equation}
holds for sufficiently large $a$ and $b$. This conjecture is confirmed by Males \cite{Males-2021}.
Gomez and Zhu \cite{Gomez-Zhu-2021} proved that \eqref{SN} holds for $t=2$ and $a,b\geq 12$.
The strict log-subadditivity for other statistics are referred to
\cite{Baker-Males-2022,Cesana-Craig-Males-2021,Hamakiotes-Kriegman-Tsai-2021}.

In this paper,
we show the strict log-subadditivity rank of overpartitions.
Recall that an overpartition of a nonnegative integer $n$ is a partition of $n$ where the first occurrence of each distinct part may be overlined.
Let $\ell(\lambda)$ denote the largest part of $\lambda$, $\#(\lambda)$ denote the number of parts of $\lambda$.
Lovejoy \cite{Lovejoy-2005} defined the rank of an overpaprtition $\lambda$ as
\begin{equation*}
  rank(\lambda)=\ell(\lambda)-\#(\lambda).
\end{equation*}
Let $\overline{N}(a,c,n)$ be the number of overpartitions of $n$ with rank congruent to $a$ modulo $c$.
Our main result as stated as follows.

\begin{thm}\label{N-con}
Given any residue $a\pmod c$, we have
\begin{align}\label{N-con-ineq}
 \N(a,c,n_1+n_2)<\N(a,c,n_1)\N(a,c,n_2)
\end{align}
if $3\leq c\leq 5$ and $n_1,n_2\geq 9$
or if $c\geq6$ and $n_1,n_2\geq M_c$, where
\[M_c=1.691\times10^{13}c^{20}\exp\left(\frac {16c^2(2c^2+\pi)}{\pi^2}\right).\]
\end{thm}

This paper is organized as follows.
In Section \ref{asy-2},
we state the asymptotic behavior of $\N(a,c,n)$ due to Ciolan.
In Section \ref{b-esti}, we estimate the upper bounds for both main terms and error terms in the asymptotic formula.
Based on the estimates, we finally prove Theorem \ref{N-con} in Section \ref{proof-thm}.

\section{Asymptotic formula of Ciolan}\label{asy-2}

Our approach to the strict log-subadditivity of overpartition rank relies on the asymptotic behavior of generating function of $\N(a,c,n)$.
In this section, we recall the asymptotic formula due to Ciolan \cite{Ciolan-2019,Ciolan-2021}.

We begin with some notations that we need.
Let
\begin{equation*}
  \omega_{h,k}:=\exp(\pi i s(h,k)),
\end{equation*}
where the Dedekind sum $s(h,k)$ is defined by
\begin{equation*}
  s(h,k):=\sum_{u~(\mathrm{mod} ~k)} \left(\left(\frac uk\right)\right)\left(\left(\frac {hu}k\right)\right).
\end{equation*}
Here $((\cdot ))$ is the sawtooth function defined by
\begin{equation*}
 ((x)):=\begin{cases}
 x-\lfloor x \rfloor-\frac 12, &\text{if} ~~x\in \mathbb{R}\setminus\mathbb{Z},\\
 0, & \text{if} ~~x\in \mathbb{Z}.
 \end{cases}
\end{equation*}
In what follows, $0\leq h<k$ are coprime integers, and $h'\in \mathbb{Z}$ is defined by the congruence $hh'\equiv 1$ (mod $k$). Further, let $0<a<c$ are coprime positive integers with $c>2$, $c_1:=\frac
c{(c,k)}$ and $k_1:=\frac k{(c,k)}$. Let the integer $0\leq l<c_1$ be the solution to $l\equiv ak_1$ (mod $c_1$).

We then introduce several Kloosterman sums. Here and throughout we write $\sum_h^{'}$ to denote summation over the integers $0\leq h<k$ that are coprime to $k$.

If $c|k$, let
\begin{equation*}
  A_{a,c,k}(n,m):=(-1)^{k_1+1} \tan\left(\frac {\pi a}c\right) {\sum_h}' \frac {\omega_{h,k}^2}{\omega_{h,k/2}}\cdot\cot\left(\frac {\pi ah'}c\right)\cdot e^{-\frac {2\pi ih'a^2k_1}c} \cdot e^{\frac {2\pi i}k (nh+mh')},
\end{equation*}
and
\begin{equation*}
  B_{a,c,k}(n,m):=-\frac 1{\sqrt{2}} \tan\left(\frac {\pi a}c\right) {\sum_h}' \frac {\omega_{h,k}^2}{\omega_{2h,k}}\cdot \frac 1{\sin\left(\frac {\pi ah'}c\right)}\cdot e^{-\frac {2\pi ih'a^2k_1}c} \cdot e^{\frac {2\pi i}k (nh+mh')}.
\end{equation*}
If $c\nmid k$ and $0<\frac l{c_1}\leq \frac 14$, let
\begin{equation*}
 D_{a,c,k}(n,m):=\frac 1{\sqrt{2}} \tan\left(\frac {\pi a}c\right) {\sum_h}' \frac {\omega_{h,k}^2}{\omega_{2h,k}} \cdot e^{\frac {2\pi i}k (nh+mh')},
\end{equation*}
and if $c\nmid k$ and $\frac 34<\frac l{c_1}< 1$, let
\begin{equation*}
 D_{a,c,k}(n,m):=-\frac 1{\sqrt{2}} \tan\left(\frac {\pi a}c\right) {\sum_h}' \frac {\omega_{h,k}^2}{\omega_{2h,k}} \cdot e^{\frac {2\pi i}k (nh+mh')}.
\end{equation*}
Finally, if $c\nmid k$, set
\begin{equation*}
  \delta_{c,k,r}:=\begin{cases}
  \frac 1{16}-\frac l{2c_1}+\frac {l^2}{c_1^2}-r\frac l{c_1}, &\text{if}~~0<\frac l{c_1}\leq \frac 14,\\[3pt]
  0,&\text{if}~~\frac 14<\frac l{c_1}\leq \frac 34,\\[3pt]
  \frac 1{16}-\frac {3l}{2c_1}+\frac {l^2}{c_1^2}+\frac 12-r\left(1-\frac l{c_1}\right), &\text{if}~~\frac 34<\frac l{c_1}< 1,
  \end{cases}
\end{equation*}
and
\begin{equation*}
  m_{a,c,k,r}:=\begin{cases}
  -\frac 1{2c_1^2}(2(ak_1-l)^2+c_1(ak_1-l)+2rc_1(ak_1-l)), &\text{if}~~0<\frac l{c_1}\leq \frac 14,\\[3pt]
  0,&\text{if}~~\frac 14<\frac l{c_1}\leq \frac 34,\\[3pt]
  -\frac 1{2c_1^2}(2(ak_1-l)^2+3c_1(ak_1-l)-2rc_1(ak_1-l)-c_1^2(2r-1)), &\text{if}~~\frac 34<\frac l{c_1}< 1.
  \end{cases}
\end{equation*}

Let $\N(m,n)$ denote the number of overpartitions of $n$ with rank $m$.
We adopt the notation $\mathcal{O}(u; q)$ in \cite{Lovejoy-2005} to denote the generating function of $\N(m,n)$, that is,
 \begin{align}\label{O-u}
\mathcal{O}(u; q):=\sum_{n\geq0}\sum_{m}\N(m, n)u^mq^n.
\end{align}
If $0<a<c$ are coprime positive integers, and if by $\zeta_{n}=e^{\frac{2 \pi i}{n}}$ we denote the primitive $n$-th root of unity, we set
\begin{align}\label{O-a/c}
\mathcal{O}\left(\frac{a}{c} ; q\right):=\mathcal{O}\left(\zeta_{c}^{a} ; q\right)=1+\sum_{n\geq 1} A\left(\frac{a}{c} ; n\right) q^{n}.
\end{align}

Ciolan obtained the following asymptotic formula for $A\left(\frac ac;n\right)$.
\begin{thm}{\rm{ \cite[ Theorem 1]{Ciolan-2019} }}\label{AA-thm}
If $0 <a <c$ are coprime positive integers with $c >2$, and $\varepsilon>0$ is arbitrary, then
\begin{align}\label{AA}
A\left(\frac ac;n\right)=&i\sqrt{\frac 2n}\sum_{\substack{1\leq k\leq \sqrt{n}\\[2pt] c|k, ~2\nmid k}}\frac {B_{a,c,k}(-n,0)}{\sqrt{k}} \cdot\sinh\left(\frac {\pi \sqrt{n}}k\right)
\nonumber\\[6pt]
&+2\sqrt{\frac 2n}\sum_{\substack{1\leq k\leq \sqrt{n}\\[2pt] c|k,~ 2\nmid k,~ c_1\neq 4\\[2pt] r\geq0,~ \delta_{c,k,r}>0}}\frac {D_{a,c,k}(-n,m_{a,c,k,r})}{\sqrt{k}} \cdot\sinh\left(\frac {4\pi \sqrt{\delta_{c,k,r}n}}k\right)+O(n^{\varepsilon}).
\end{align}
\end{thm}

To obtain the inequality of $\N(a,c,n)$, Ciolan \cite{Ciolan-2021} bounded the main term and the error term in above theorem.
Here we list these bounds respectively.
\begin{itemize}
  \item Bounds in the main terms:
  Ciolan \cite[\S~4.2]{Ciolan-2021} gave upper bounds for Kloosterman sums and hyperbolic sine function in \eqref{AA}.
  \begin{align}\label{Main-c-1}
 &\left|B_{j,c,k}(-n,0)\right|\leq \cot\left(\frac {\pi}{2c}\right)\frac {2k\left(1+\log\left(\frac {c-1}2\right)\right)}{\pi\left(1-\frac{\pi^2}{24}\right)},
\\[6pt] \label{Main-c-2}
&\left|D_{j,c,k}(-n,m_{j,c,k,r})\right|\leq \frac k{\sqrt{2}}\cot\left(\frac {\pi}{2c}\right),
\\[6pt] \label{Main-c-3}
 &\sinh\left(\frac {4\pi \sqrt{\delta_{c,k,r}n}}k\right)\leq \sinh \left(\pi \sqrt{n}\left(1-\frac 4c\right)\right).
\end{align}
  \item Bounds in the error terms: the error term $O(n^{\varepsilon})$ is constituted of $S_1, S_2, \ldots, S_8$, the error terms of $\sum_1, \sum_2, \ldots, \sum_8$ respectively.
Here we strictly follows the notations of Ciolan in \cite[\S~4.3]{Ciolan-2021}.
Let $\p(n)$ denote the number of overpartitions of $n$. Then
\begin{equation}
\begin{split}
&S_1<4C_3e^{2\pi}\cot\left(\frac {\pi}{2c}\right)\sum_k k^{-\frac 12},
\qquad S_2<4C_1e^{2\pi}\sqrt{2}\cot\left(\frac {\pi}{2c}\right)\sum_k k^{-\frac 12},
\\[6pt]
&S_3<2C_4e^{2\pi}\cot\left(\frac {\pi}{2c}\right)\sum_k k^{-\frac 12},
\qquad S_4<C_5e^{2\pi}\cot\left(\frac {\pi}{2c}\right)\sum_k k^{-\frac 12},
\\[6pt]
&S_5<C_2e^{2\pi}\sqrt{2}\cot\left(\frac {\pi}{2c}\right)\sum_k k^{-\frac 12},
\quad S_6<C_2e^{2\pi}\frac 1{\sqrt{2}}\cot\left(\frac {\pi}{2c}\right)\sum_k k^{-\frac 12},
\\[6pt]
&S_7<\frac {n^{\frac 34}\log\left(\frac n4\right)}{2\pi\left(1-\frac{\pi^2}{24}\right)\sin\left(\frac {\pi}c\right)},
\qquad\qquad S_8<\frac {n^{\frac 34}\log\left(\frac n4\right)}{2\pi\left(1-\frac{\pi^2}{24}\right)\sin\left(\frac {\pi}c\right)},
\end{split}
\label{error-c}
\end{equation}
where
\begin{align}\label{C2}
&C_1=\sum_{r\geq 1}\p(r)\left(e^{-\frac {(16r-1)\pi}{16}}+e^{-\frac {(16r+7)\pi}{16}}\right),
\quad
C_2=2\sum_{r\geq 1}\p(r)e^{-\frac {(c^2-8)\pi r}{16c^2}},
\\[3pt] \label{C4}
&C_3=\sum_{r\geq 1}\p(r)e^{-\pi r},
\quad
C_4=\sum_{r\geq 1}\p(r)e^{-\frac {\pi r}{2c^2}},
\quad
C_5=\sum_{r\geq 1}\p(r)e^{-\frac {\pi (2r+1)}8}.
 \end{align}
To our need, we adopt the following notations. The errors arising from integrating over the remaining parts of the interval are denoted by $S_{2err}, S_{5err}, S_{6err}$,
and errors introduced by integrating along a smaller arc are denoted by $I_{2err}, I_{5err}, I_{6err}$.
Ciolan \cite[\S~4.3]{Ciolan-2021} provided the following bounds on each of those pieces:
\begin{align}
S_{2err}&<\sqrt{2}e^{2\pi+\frac {\pi}8}\cot\left(\frac {\pi}{2c}\right)n^{-\frac 12}\frac {1+\log\left(\frac {c-1}2\right)}{\pi\left(1-\frac{\pi^2}{24}\right)}\sum_k k^{-\frac 12},
\nonumber\\[6pt]
\label{error-c2}
S_{5err}&<2\sqrt{2}e^{2\pi+\frac {\pi}8}\cot\left(\frac {\pi}{2c}\right)n^{-\frac 12}\frac {1+\log\left(\frac {c-1}2\right)}{\pi\left(1-\frac{\pi^2}{24}\right)}\sum_k k^{-\frac 12},
\\[6pt]
\nonumber
S_{6err}&<\sqrt{2}e^{2\pi+\frac {\pi}8}\cot\left(\frac {\pi}{2c}\right)n^{-\frac 12}\frac {1+\log\left(\frac {c-1}2\right)}{\pi\left(1-\frac{\pi^2}{24}\right)}\sum_k k^{-\frac 12}
\end{align}
and
\begin{align}
I_{2err}&<2\sqrt{2}\left(\frac 43+2^{\frac 54}\right)e^{2\pi+\frac {\pi}8}\cot\left(\frac {\pi}{2c}\right)\frac {1+\log\left(\frac {c-1}2\right)}{\pi\left(1-\frac{\pi^2}{24}\right)}n^{\frac 14},
\nonumber\\[6pt]
\label{error-c3}
I_{5err}&<4\sqrt{2}\left(\frac 43+2^{\frac 54}\right)e^{2\pi+\frac {\pi}8}\cot\left(\frac {\pi}{2c}\right)\frac {1+\log\left(\frac {c-1}2\right)}{\pi\left(1-\frac{\pi^2}{24}\right)}n^{\frac 14},
\\[6pt]
\nonumber
I_{6err}&<2\sqrt{2}\left(\frac 43+2^{\frac 54}\right)e^{2\pi+\frac {\pi}8}\cot\left(\frac {\pi}{2c}\right)\frac {1+\log\left(\frac {c-1}2\right)}{\pi\left(1-\frac{\pi^2}{24}\right)}n^{\frac 14}.
\end{align}
\end{itemize}

\section{Bounds for $\N(a,c,n)$}\label{b-esti}
In this section,
we show that $\N(a,c,n)$ can be bounded by a constant multiple of $\p(n)$ for each $c\geq3$.

\begin{thm}\label{bound-N}
We have
\begin{align}\label{bound-N-3}
&0.0019~\p(n)<\N(a,3,n)<0.6648~\p(n),~~n\geq2089,
\\[6pt] \label{bound-N-4}
& 0.0091~\p(n)<\N(a,4,n)<0.4909~\p(n),~~n\geq 272,
\\[6pt] \label{bound-N-5}
&0.0103~\p(n)<\N(a,5,n)<0.3897~\p(n),~~n\geq 449.
\end{align}
For $c\geq6$ we define the constant
\begin{equation}\label{Mc-bound}
  M_c:=1.691\times10^{13}c^{20}\exp\left(\frac {16c^2(2c^2+\pi)}{\pi^2}\right),
\end{equation}
we have
\begin{align}\label{bound-N-c}
 \left(\frac 1{2c}\right)\p(n)< \N(a,c,n)< \left(\frac 3{2c}\right)\p(n),~~n\geq M_c.
\end{align}
\end{thm}

Our main mission of this section is to prove the above theorem.
We sketch our strategy here.
Setting $u=\zeta_c^a$ in \eqref{O-u}, we have
\begin{align}\label{N-g}
\mathcal{O}\left(\frac{a}{c} ; q\right)=\sum_{n\geq0} \sum_{m}\N(m, n)\zeta_c^{am}q^n.
\end{align}
Comparing the coefficient of $q^n$ in \eqref{O-a/c}, \eqref{N-g} and with the orthogonality of the roots of unity, we obtain
\begin{equation}\label{N-id}
  \N(a,c,n)=\frac{1}{c}\p(n)+\frac 1c\sum_{1\leq j\leq c-1}\zeta_c^{-aj}A\left(\frac jc;n\right).
\end{equation}
The second summation in right-hand side of \eqref{N-id} is denoted by $\R(a,c,n)$, that is,
\begin{align}\label{N-R}
\N(a,c,n)=\frac{1}{c}\p(n)+\R(a,c,n).
\end{align}
Due to \eqref{N-R}, we need to estimate $\R(a,c,n)$ and $\p(n)$ separately.
For $\R(a,c,n)$,
we need to give a bound to the main terms as well as the error terms in view of Theorem \ref{AA-thm}.
These are completed in Theorem \ref{Main-term-R} and Theorem \ref{Error-term-R} respectively.
After then, we give an upper bound and lower bound for $\p(n)$ in Theorem \ref{p-bound-thm} based on the asymptotic formula of Engel \cite{Engel-2017}.

From the view of Theorem \ref{AA-thm},
we have the following estimates to bound the main terms of $\R(a,c,n)$.

\begin{lem}
For integer $c >2$, then we have
\begin{align}\label{Main-1}
&\left|B_{j,c,k}(-n,0)\right|\leq 0.3444 k c^2,
\\[6pt] \label{Main-2}
&\left|D_{j,c,k}(-n,m_{j,c,k,r})\right|\leq 0.4503kc,
\\[6pt] \label{Main-3}
&\sinh\left(\frac {4\pi \sqrt{\delta_{c,k,r}n}}k\right)\leq
\frac{1}{2}e^{\pi \sqrt{n}\left(1-\frac 4c\right)}.
\end{align}
\end{lem}

\begin{proof}
We shall estimate \eqref{Main-1}. Using the fact in \cite[P. 159]{Bullen-1998},
\begin{align}\label{log}
\log x&\leq a(x^{1/a}-1),~~~~a,x>0,
\end{align}
we have
\begin{align}\label{estim-1}
\frac {1+\log\left(\frac {c-1}2\right)}{\pi\left(1-\frac{\pi^2}{24}\right)}< 0.2704c.
\end{align}
Substituting \eqref{estim-1} and
\begin{align}\label{cot}
\cot\left(\frac {\pi}{2c}\right)\leq \frac {2c}{\pi}< 0.6367c
\end{align}
into \eqref{Main-c-1}, we get the first inequality \eqref{Main-1}.
By means of the estimates in \eqref{Main-c-2} and \eqref{cot}, we arrive at the second inequality \eqref{Main-2}.
The third inequality \eqref{Main-3} immediately follows from the definition of hyperbolic sine function. This completes the proof.
\end{proof}

We can give a bound for the main term of $\R(a,c,n)$.

\begin{thm}\label{Main-term-R}
A bound for the main term of $\R(a,c,n)$ is
\[0.1624 e^{\frac {\pi \sqrt{n}}c}n^{\frac 14}c+(0.0266c+0.2123)e^{\pi \sqrt{n}\left(1-\frac 4c\right)}n^{\frac 14}c.\]
\end{thm}

\begin{proof}
Immediately, the main term of \eqref{AA} gives the main term of $\R(a,c,n)$
\begin{align}\label{N}
&\frac 1c \sum_{1\leq j\leq c-1}\zeta_c^{-aj}  i\sqrt{\frac 2n}\sum_{\substack{1\leq k\leq \sqrt{n}\\[2pt] c|k, 2\nmid k}}\frac {B_{j,c,k}(-n,0)}{\sqrt{k}} \cdot\sinh\left(\frac {\pi \sqrt{n}}k\right)
\nonumber\\[6pt]
&\quad+\frac 1c \sum_{1\leq j\leq c-1}\zeta_c^{-aj} 2\sqrt{\frac 2n}\sum_{\substack{1\leq k\leq \sqrt{n}\\[2pt]
c|k, 2\nmid k, c_1\neq 4\\[2pt] r\geq0, \delta_{c,k,r}>0}}\frac {D_{j,c,k}(-n,m_{j,c,k,r})}{\sqrt{k}} \cdot\sinh\left(\frac {4\pi \sqrt{\delta_{c,k,r}n}}k\right).
\end{align}
Let $G_1$ and $G_2$ denote the summations of \eqref{N}, respectively.
By the bounds in \eqref{Main-1}, we have
\begin{align}\label{main1}
&\left|G_1\right|\leq \sqrt{\frac 2n}\sinh\left(\frac {\pi \sqrt{n}}c\right)
\sum_{\substack{1\leq k\leq \sqrt{n}\\[2pt]c|k}}0.3444k^{\frac 12}c^2
\leq 0.1624 e^{\frac {\pi \sqrt{n}}c}n^{\frac 14}c,
\end{align}
where the last inequality is due to
\begin{equation*}
\sum_{\substack{1\leq k\leq \sqrt{n}\\[2pt]c|k}}k^{\frac 12}
\leq c^{\frac{1}{2}}\sum_{1\leq m\leq \lfloor\frac{N}{c}\rfloor}m^{\frac{1}{2}}
\leq c^{\frac{1}{2}}\int_1^{\lfloor\frac{N}{c}\rfloor}
x^{\frac{1}{2}}dx
\leq\frac 2{3c}n^{\frac 34}.
\end{equation*}

Now we consider the bound of $|G_2|$.
For fixed $k$, it is easy to see that
\begin{equation}\label{1-bounds}
  \sum_{\substack{r\geq0 \\[2pt]c|k, 2\nmid k, c_1\neq 4\\[2pt] \delta_{c,k,r}>0}} 1\leq \frac {c+8}{16}.
\end{equation}
Then by \eqref{Main-2}, \eqref{Main-3} and \eqref{1-bounds}, we get
\begin{align*}
\left|G_2\right|
&\leq 0.4503c\sqrt{\frac 2n}e^{\pi \sqrt{n}\left(1-\frac 4c\right)}\frac {c+8}{16}\sum_{\substack{1\leq k\leq \sqrt{n}\\[2pt] c\nmid k}}k^{\frac 12}.
\end{align*}
Since
\begin{equation*}
 \sum_{\substack{1\leq k\leq \sqrt{n}\\[2pt] c\nmid k}}k^{\frac 12}\leq\int_1^{\sqrt{n}} x^{\frac 12}dx\leq \frac 23 n^{\frac 34},
\end{equation*}
we have
\begin{align}\label{main3}
\left|G_2\right|
&\leq 0.4503c\sqrt{\frac 2n}e^{\pi \sqrt{n}\left(1-\frac 4c\right)}\cdot \frac {c+8}{16} \cdot \frac 23 n^{\frac 34}
\nonumber\\[6pt]
&\leq (0.0266c+0.2123)e^{\pi \sqrt{n}\left(1-\frac 4c\right)}n^{\frac 14}c.
\end{align}
Plugging \eqref{main1} and \eqref{main3} into \eqref{N}, we complete the proof.
\end{proof}

Next we shall give a bound for the error terms of $\R(a,c,n)$.
To this end, we need to analyse the terms in \eqref{error-c}, \eqref{error-c2} and \eqref{error-c3} more detailedly.

\begin{lem}\label{error}
For integer $c >2$, then we have
\begin{align*}
S_1 &\leq 1496.9n^{\frac 14}c,
\quad S_2 \leq 3111.36n^{\frac 14}c,
\quad~~ S_3 \leq 1363.79\overline{C}_4n^{\frac 14}c,
\\[3pt]
S_4 &\leq 82469.8n^{\frac 14}c,
\quad S_5 \leq 964.35\overline{C}_2 n^{\frac 14}c,
\quad S_6 \leq 482.18\overline{C}_2n^{\frac 14}c,
\\[3pt]
S_7 &\leq0.9093n^{\frac 78}c,
\quad~~ S_8  \leq0.9093n^{\frac 78}c
\end{align*}
and
\begin{align*}
S_{2err}&\leq 386.18n^{-\frac 14}c^2,
\quad S_{5err}\leq 772.36n^{-\frac 14}c^2,
\quad S_{6err}\leq 386.18n^{-\frac 14}c^2,
\\[3pt]
I_{2err}&\leq 1433.39n^{\frac 14}c^2,
\quad I_{5err}\leq 2866.78n^{\frac 14}c^2,
\quad~~~ I_{6err}\leq 1433.39n^{\frac 14}c^2,
\end{align*}
where
\begin{align*}
\overline{C}_2=2\exp\left(\frac {32c^2(16c^2+(c^2-8)\pi)} {\pi^2(c^2-8)^2 }\right),
\quad
\overline{C}_4=\exp\left(\frac {4c^2(2c^2+\pi)}{\pi^2}\right).
\end{align*}
\end{lem}

\begin{proof}
First, we give the estimations of $C_1,C_2, C_3,C_4, C_5$ in \eqref{C2} and \eqref{C4}.
By the fact
\begin{align}\label{over-est}
\p(n)< e^{\pi\sqrt{n}},
\end{align}
it is easy to check that
\begin{align}\label{error-c-1}
C_1\leq 0.8066,
\quad
C_3\leq 0.5488,
\quad
C_5\leq 120.942.
\end{align}

To estimate $C_2$, we claim that
\begin{align}\label{overp-b}
  \sum_{n\geq 0}\p(n)e^{-2\pi ny}\leq \exp\left(\frac {2e^{-2\pi y}}{(1-e^{-2\pi y})^2}\right),
\end{align}
where $y$ is a positive real number.
Recall that the generating function of $\p(n)$ is given by Corteel and Lovejoy \cite{Corteel-Lovejoy-2004}
\begin{equation}\label{overp-gener}
  F(q):=\sum_{n\geq0}\p(n)q^n=\prod_{n\geq1}\frac{(1+q^n)}{(1-q^n)},
\end{equation}
where $q=e^{2\pi iz}$ with $z=x+iy$.
Taking logarithm of both sides of \eqref{overp-gener}, we have
\begin{align*}
  \log F(q)
  &=\sum_{n\geq1} \left(\log (1+q^n)-\log (1-q^n)\right)
  \\[5pt]
  &=\sum_{n\geq1} \sum_{m\geq1} \left(\frac {q^{nm}}m-\frac {(-1)^m q^{nm}}m\right)\\[5pt]
  &<\frac {2q}{(1-q)^2}.
\end{align*}
Hence
\begin{equation*}
  F(|q|)\leq \exp \left(\frac {2|q|}{(1-|q|)^2}\right),
\end{equation*}
together with $|q|=e^{-2\pi y}$, which gives \eqref{overp-b} as claimed.
Since integer $c>2$, $c^2-8>0$.
Take $y=(c^2-8)/32c^2$ in \eqref{overp-b} and we have
\begin{align*}
   C_2&\leq 2\exp\left(\frac {2e^{-\frac {(c^2-8)\pi }{16c^2}}}{\left(1-e^{\frac {(c^2-8)\pi }{16c^2}}\right)^2}\right).
\end{align*}
For $x>-1$, it can be easily derived from
\begin{equation*}
  \frac x{1+x}<1-e^{-x}<x
\end{equation*}
that
\begin{equation}\label{e}
  \frac {e^{-x}}{(1-e^{-x})^2}< \frac {1+x}{x^2}.
\end{equation}
Therefore
\begin{align}\label{C2-b}
   C_2\leq 2\exp\left(\frac {32c^2(16c^2+(c^2-8)\pi)} {\pi^2(c^2-8)^2 }\right):=\overline{C}_2.
\end{align}

By the same argument as above, one can derive
\begin{equation}\label{C4-b}
  C_4\leq \exp\left(\frac {4c^2(2c^2+\pi)}{\pi^2}\right):=\overline{C}_4
\end{equation}
from \eqref{overp-b} and \eqref{e}.

Since
\begin{align}\label{sum-k}
\sum_k k^{-\frac 12}=\int_1^{\sqrt{n}} x^{-\frac 12}dx\leq2n^{\frac 14}.
\end{align}
Substituting \eqref{cot}, \eqref{sum-k} and the bounds for $C_1, \ldots, C_5$ as given in \eqref{error-c-1}, \eqref{C2-b} and \eqref{C4-b} into \eqref{error-c}, we finally derive estimates for $S_1,\ldots,S_6$ in Lemma \ref{error}.

As for $S_7$ and $S_8$, first we note that
when $x\in\left(0,\frac{\pi}{2}\right)$,
\[\sin x\geq \frac {2x}{\pi}.\]
Then we have
\begin{align}\label{est-sin}
\sin\left(\frac {\pi}c\right)\geq \frac 2c .
\end{align}
Meanwhile, setting $x=\frac{n}{4}$ and $a=8$ in \eqref{log}, we have
\begin{equation}\label{est-log}
  \log\left(\frac n4\right)\leq 8\left(\frac n4\right)^{\frac 18}.
\end{equation}
Thus  $S_7$ and $S_8$ is bounded by $0.9093n^{\frac 78}c$ as we
substitute \eqref{est-sin} and \eqref{est-log} into \eqref{error-c}.
The asymptotic estimates $S_{2err}$, $S_{5err}$, $S_{6err}$, $I_{2err}$, $I_{5err}$ and $I_{6err}$ easily follow from  \eqref{estim-1}, \eqref{cot} and \eqref{sum-k}. This completes the proof.
\end{proof}

A bound for the error terms of $\R(a,c,n)$ immediately follows from the above lemma.

\begin{thm}\label{Error-term-R}
The bound for the error term of $\R(a,c,n)$ is
\begin{align}\label{err}
  1544.72n^{-\frac 14}c^2+87078.1n^{\frac 14}c&+5733.56n^{\frac 14}c^2
  +1.8186n^{\frac 78}c
  \nonumber\\[6pt]
  &\quad+1363.79\overline{C}_4n^{\frac 14}c
  +1446.53 \overline{C}_2n^{\frac 14}c.
\end{align}
\end{thm}

\begin{proof}
Adding the estimates in Lemma \ref{error} together and applying to \eqref{AA}, we get the bound for the error term of $\R(a,c,n)$ as desired.
\end{proof}

The following bounds for $\p(n)$ serves as a connection between $\p(n)$ and $\R(a,c,n)$.

\begin{thm}\label{p-bound-thm}
We have
\begin{equation}\label{p-bound}
  \frac{1}{8n}\left(1-\frac{1}{\sqrt{n}}\right)e^{\pi\sqrt{n}}\leq \p(n)\leq \frac{1}{8n}\left(1+\frac{1}{\sqrt{n}}\right)e^{\pi\sqrt{n}}.
\end{equation}
\end{thm}

\begin{proof}
Recall that Engel \cite{Engel-2017} provided an error term for the overpartition function
\begin{align*}
\p(n)=\frac{1}{2\pi}\sum_{1\leq k\leq N\atop 2\nmid k}\sqrt{k}\sum_{0\leq h\leq k \atop (h,k)=1}
\frac{\omega(h,k)^2}{\omega(2h,k)}e^{-\frac{2\pi inh}{k}}\frac{\mathrm{d}}{\mathrm{d}n}
\left(\frac{\sinh\left(\frac{\pi\sqrt{n}}{k}\right)}{\sqrt{n}}\right)+R_2(n,N),
\end{align*}
where
\begin{align*}
|R_2(n,N)|\leq \frac{N^{\frac{5}{2}}}{\pi n^{\frac{3}{2}}}
\sinh\left(\frac{\pi\sqrt{n}}{N}\right).
\end{align*}
In particular, when $N=3$, we have
\begin{align}\label{overlinep-asym-1}
\p(n)=\frac{1}{8n}\left[\left(1+\frac{1}{\pi\sqrt{n}}\right)e^{-\pi\sqrt{n}}+\left(1-\frac{1}{\pi\sqrt{n}}\right)e^{\pi\sqrt{n}}\right]
+R_2(n,3),
\end{align}
where
\begin{align}\label{R_2(n,3)}
|R_2(n,3)|\leq \frac{3^{\frac{5}{2}}}{\pi n^{\frac{3}{2}}}
\sinh\left(\frac{\pi\sqrt{n}}{3}\right)\leq \frac{3^{\frac{5}{2}}e^{\frac{\pi\sqrt{n}}{3}}}{2\pi n^{\frac{3}{2}}}
\leq\frac{\pi e^{\pi\sqrt{n}}}{16\pi n^{\frac{3}{2}}}.
\end{align}
By \eqref{overlinep-asym-1}, we get the following upper bound for $\p(n)$
\begin{align*}
\p(n)\leq\frac{1}{8n}\left(1+\frac{1}{\pi\sqrt{n}}\right)e^{\pi\sqrt{n}}+\frac{\pi e^{\pi\sqrt{n}}}{16\pi n^{\frac{3}{2}}}.
\end{align*}
The right hand side of above inequality can be rewritten as
\begin{align*}
\frac{1}{8n}\left(1+\frac{1}{\pi\sqrt{n}}\right)e^{\pi\sqrt{n}}+\frac{\pi e^{\pi\sqrt{n}}}{16\pi n^{\frac{3}{2}}}
=\frac{1}{8n}\left(1+\frac{2+\pi}{ 2\pi\sqrt{n}}\right)e^{\pi\sqrt{n}},
\end{align*}
which arrives at the right-hand side of \eqref{p-bound}.
The left hand-side of \eqref{p-bound} can be verified easily. This completes the proof.
\end{proof}

We are now ready to complete the proof of Theorem \ref{bound-N}.

\begin{proof}[Proof of Theorem \ref{bound-N}]
Combining Theorems \ref{Main-term-R} and \ref{Error-term-R}, we have
\begin{align}\label{R-bound}
  \left|\R(a,c,n)\right|
  \leq &\bigg| 0.1624 e^{\frac {\pi\sqrt{n}}c}n^{\frac 14}c+(0.0266c+0.2123)e^{\pi\sqrt{n}\left(1-\frac 4c\right)}n^{\frac 14}c
 \notag \\[6pt]
  &\quad+1544.72n^{-\frac 14}c^2+87078.1n^{\frac 14}c+5733.56n^{\frac 14}c^2+1.8186n^{\frac 78}c
  \notag\\[6pt]
  &\quad+1363.79\overline{C}_4n^{\frac 14}c+1446.53 \overline{C}_2n^{\frac 14}c\bigg|.
\end{align}
We note that for $n\geq2$,
\begin{equation*}
\frac {\sqrt{n}}{\sqrt{n}-1}\leq \frac {\sqrt{2}}{\sqrt{2}-1}\leq 3.415.
\end{equation*}
By Theorem \ref{p-bound-thm}, the lower bound of $\p(n)^{-1}$ can be estimated as
\begin{align*}
\p(n)^{-1}\leq 8n\frac {\sqrt{n}}{\sqrt{n}-1}e^{-\pi\sqrt{n}}
\leq 27.32ne^{-\pi\sqrt{n}}.
\end{align*}
Combining above inequality,
we have for $c>2$
\begin{align}\label{R-bound-1}
\left|\frac{\R(a,c,n)}{\p(n)}\right|\leq & 4.44c e^{\left(\frac 1c -1\right)\pi\sqrt{n}}n^{\frac 54}+(0.73c+5.81)ce^{-\frac 4c \pi\sqrt{n}}n^{\frac 54}+e^{-\pi\sqrt{n}}\left[42201.8c^2 n^{\frac 34}\right.
   \notag \\
  &+156641c^2n^{\frac 54}+2.379\times10^6 cn^{\frac 54}+49.69c n^{\frac {15}8}+37258.7c\overline{C}_4n^{\frac 54}
  \notag \\
  &\left.
  +39519.2c \overline{C}_2n^{\frac 54}\right].
\end{align}

For $c=3$, recalling the definition of $C_2$ in \eqref{C2}, then using \eqref{over-est}, we compute that $C_2< 4.5303\times10^{52}$. Similarly, using \eqref{C4}, we have $C_4< 1.0535\times10^{8}$. So replacing $\overline{C}_2$ and $ \overline{C}_4$ in \eqref{R-bound-1} with $ 4.5303\times10^{52}$ and $ 1.0535\times10^{8}$,
we obtain that
\begin{align}\label{R-bound-3}
  \left|\frac{\R(a,3,n)}{\p(n)}\right|
  \leq & 13.32 e^{\frac{-2\pi\sqrt{n}}{3}}n^{\frac 54}+24e^{-\frac 43 \pi\sqrt{n}}n^{\frac 54}+e^{-\pi\sqrt{n}}\left(379816.2 n^{\frac 34}\right. \notag\\
  &\quad \quad \quad\left.+5.3711\times 10^{57}n^{\frac 54} +149.07n^{\frac {15}8}\right):=R_3.
\end{align}
Plugging \eqref{R-bound-3} into \eqref{N-R}, we have
\begin{align}\label{bound-p-R-3}
\p(n)\left(\frac 13-R_3\right)<\N(a,3,n)<\p(n)\left(\frac 13+R_3\right).
\end{align}
Note that $R_3$ is decreasing for $n>1$, then we have
\begin{align}\label{bound-R-3}
R_3<0.33142,~~n\geq2089.
\end{align}
Combining \eqref{bound-p-R-3} and \eqref{bound-R-3}, we are led to \eqref{bound-N-3}.

Similar to $c=3$, we first confirm the estimations of $C_2$ and $C_4$.
When $c=4$, $C_2<2.977\times10^{13}$ and $C_4<1.4885\times10^{13}$, for $c=5$, we have $C_2<2.4221\times10^{10}$ and $C_4<4.0102\times10^{19}$.
Then we obtain
\begin{align}\label{R-bound-4}
   \left|\frac{\R(a,4,n)}{\p(n)}\right|
  &\leq 17.76 e^{\frac{-3\pi\sqrt{n}}{4}}n^{\frac 54}+e^{-\pi\sqrt{n}}\left(675228.9 n^{\frac 34}\right.
  \nonumber\\[6pt]
 &\left.\qquad\qquad+6.9244\times 10^{18}n^{\frac 54} +198.76n^{\frac {15}8}\right):=R_4,
\end{align}
and
\begin{align}\label{R-bound-5}
   \left|\frac{\R(a,5,n)}{\p(n)}\right|
  &\leq 69.5 e^{\frac{-4\pi\sqrt{n}}{5}}n^{\frac 54}+e^{-\pi\sqrt{n}}\left(1.0551\times10^6 n^{\frac 34}\right.
  \nonumber\\[6pt]
 &\left.\qquad\qquad
  +7.4708\times 10^{24}n^{\frac 54} +248.45n^{\frac {15}8}\right):=R_5.
\end{align}
Substituting \eqref{R-bound-4} and \eqref{R-bound-5} into \eqref{N-R}, respectively, we have
\begin{align}\label{bound-p-R-4}
&\p(n)\left(\frac 14-R_4\right)<\N(a,4,n)<\p(n)\left(\frac 14+R_4\right),
\\[6pt] \label{bound-p-R-5}
&\p(n)\left(\frac 15-R_5\right)<\N(a,5,n)<\p(n)\left(\frac 15+R_5\right).
\end{align}
Since
\begin{align}\label{bound-R-4}
&R_4<0.24084,~~n\geq272,
\\[6pt] \label{bound-R-5}
&R_5<0.1897,~~~~n\geq449.
\end{align}
Applying the estimates \eqref{bound-R-4} into \eqref{bound-p-R-4} and \eqref{bound-R-5} into \eqref{bound-p-R-5}, we reach \eqref{bound-N-4} and \eqref{bound-N-5}.

Next, we deal with the case $c\geq6$. Here we combine all the terms in \eqref{R-bound-1} except the term of $49.69ce^{-\pi \sqrt{n}}n^{\frac {15}8}$ into an upper bound.
We take the sum of their coefficients, the highest order exponential,
and the highest power of $n$ from the all terms and put them together in one term.
 For $c\geq6$, $4/c<1-1/c$, so we have
\begin{align}\label{R-b6}
  \left|\frac{\R(a,c,n)}{\p(n)}\right|\leq 37259c \exp\left(\frac {4c^2(2c^2+\pi)}{\pi^2}\right)e^{-\frac 4c \pi \sqrt{n}}n^{\frac 54}+49.69ce^{-\pi \sqrt{n}}n^{\frac {15}8}:=R_c.
\end{align}
In view of \eqref{N-R} and \eqref{R-b6}, we get the following bounds of $\N(a,c,n)$
\begin{align}\label{bound-p-R-c}
\p(n)\left(\frac 1c-R_c\right)<\N(a,c,n)<\p(n)\left(\frac 1c+R_c\right).
\end{align}
We claim that
\begin{align}\label{R}
R_c<\frac 1{2c},~~n\geq M_c,
\end{align}
where $M_c$ is defined in \eqref{Mc-bound}.

To verify \eqref{R}, we only need to show that
\begin{align}\label{R-1}
  37259c \exp\left(\frac {4c^2(2c^2+\pi)}{\pi^2}\right)e^{-\frac 4c \pi \sqrt{n}}n^{\frac 54}<\frac 1{4c}
\end{align}
and
\begin{align}\label{R-2}
  49.69ce^{-\pi \sqrt{n}}n^{\frac {15}8}&<\frac 1{4c},
\end{align}
which are equivalent to prove that
\begin{equation}\label{bound-c-1}
  \frac {e^{\frac 4c \pi \sqrt{n}}}{n^{\frac 54}}>4\times 37259c^2 \exp\left(\frac {4c^2(2c^2+\pi)}{\pi^2}\right)
\end{equation}
and
\begin{equation}\label{bound-c-2}
  \frac {e^{\pi \sqrt{n}}}{n^{\frac {15}8}}>4\times 49.69c^2.
\end{equation}
In addiction, we recall the following inequality\cite[P. 76]{Mitrinovic-1964}
\begin{equation*}
  e^x>\left(1+\frac xy \right)^y,~~~x, y>0.
\end{equation*}
Hence, letting $x=\frac 4c \pi \sqrt{n}$, $y=3$ and $x=\pi \sqrt{n}$, $y=4$ respectively, we obtain
\begin{align}\label{bound-c-3}
&\frac {e^{\frac 4c \pi \sqrt{n}}}{n^{\frac 54}}>\frac 1{n^{\frac 54}}\left(1+\frac {4\pi \sqrt{n}}{3c}\right)^3>\frac {4^3\pi^3}{3^3c^3}n^{\frac 14}
\end{align}
and
\begin{align}\label{bound-c-4}
&\frac {e^{\pi \sqrt{n}}}{n^{\frac {15}8}}> \frac 1{n^{\frac {15}8}}\left(1+\frac {\pi \sqrt{n}}{4}\right)^4>\frac {\pi^4}{4^4}n^{\frac 18}.
\end{align}
By \eqref{bound-c-3} and \eqref{bound-c-4}, we know that the inequalities \eqref{bound-c-1} and \eqref{bound-c-2} hold when $n\geq M_c$ and $n\geq {M_c}'$, respectively, where
\begin{align*}
  {M_c}'&:=5.544\times10^{21}c^{16}.
\end{align*}
Furthermore, we see that \eqref{R-1} and \eqref{R-2} hold when $n\geq M_c$ and $n\geq {M_c}'$, respectively.
Therefore, \eqref{R} is verified for $n\geq \max\{M_c, {M_c}'\}=M_c$.
Plugging \eqref{R} into \eqref{bound-p-R-c}, we complete the proof.
\end{proof}

\section{Strict log-subadditivity of $\N(a,c,n)$}\label{proof-thm}
In this section, we present a proof of Theorem \ref{N-con} based on the intermediate inequalities in the previous sections.
\begin{proof}[Proof of Theorem \ref{N-con}]
We begin with the case of $c=3$.
Substituting \eqref{p-bound} into \eqref{bound-N-3}, we obtain that for $n\geq2089$
\begin{equation}\label{N4-bound}
  (0.0019)\frac{1}{8n}\left(1-\frac{1}{\sqrt{n}}\right)e^{\pi \sqrt{n}}<\N(a,3,n)<(0.6648)\frac{1}{8n}\left(1+\frac{1}{\sqrt{n}}\right)e^{\pi \sqrt{n}}.
\end{equation}
Letting $n_2=Cn_1$ for some $C\geq1$, then by \eqref{N4-bound}, we have
\begin{align*}
  \N(a,3;n_1)\N(a,3;n_2)>(0.0019)^2\frac{1}{64Cn_1^2}\left(1-\frac{1}{\sqrt{n_1}}\right)
  \left(1-\frac{1}{\sqrt{Cn_1}}\right)e^{\pi (\sqrt{n_1}+\sqrt{Cn_1})},
\end{align*}
and
\begin{align*}
  \N(a,3;n_1+n_2)<(0.6648)\frac{1}{8(n_1+Cn_1)}\left(1+\frac{1}{\sqrt{n_1+Cn_1}}\right)e^{\pi \sqrt{n_1+Cn_1}}.
\end{align*}
It is sufficient to show that
\begin{equation*}
  T_{n_1}(C)>\log \left(V_{n_1}(C)\right)+\log \left(S_{n_1}(C)\right),
\end{equation*}
where
\begin{align*}
T_{n_1}(C)&:=\pi (\sqrt{n_1}+\sqrt{Cn_1})-\pi \sqrt{n_1+Cn_1},
\\[16pt]
V_{n_1}(C)&:=\frac {0.6648\times8Cn_1}{(0.0019)^2(C+1)},
\\[6pt]
S_{n_1}(C)&:=\frac {1+\frac{1}{\sqrt{n_1+Cn_1}}}{\left(1-\frac{1}{\sqrt{n_1}}\right)\left(1-\frac{1}{\sqrt{Cn_1}}\right)}.
\end{align*}
As function of $C$, it can be shown that $T_{n_1}(C)$ is increasing and $S_{n_1}(C)$ is decreasing for $C\geq1$, and combined with
\begin{equation*}
  V_{n_1}(C)<\frac {0.6648\times8n_1}{(0.0019)^2},
\end{equation*}
it suffices to show that
\begin{align}\label{T4}
   \notag
   T_{n_1}(1)&=2\pi \sqrt{n_1}-\pi \sqrt{2n_1}\\
   &>\log \left(\frac {0.6648\times8n_1}{(0.0019)^2}\right)+\log \left(\frac {1+\frac{1}{\sqrt{2n_1}}}{\left(1-\frac{1}{\sqrt{n_1}}\right)^2}\right).
\end{align}
By a short computation, we find that \eqref{T4} holds for all $n_1\geq 109$.
Hence, if $n_1, n_2\geq 2089$, we have
\begin{align*}
 \N(a,3,n_1+n_2)<\N(a,3,n_1)\N(a,3,n_2).
\end{align*}
It is routine to check that \eqref{N-con-ineq} is true for $c=3$ and $9\leq n\leq2088$.

The proofs of $c=4,5$ are similar to that of $c=3$, and hence, they are omitted.

We proceed to prove the case for $c\geq6$.
Plugging \eqref{p-bound} into \eqref{bound-N-c}, we obtain that
\begin{equation}\label{N5-bound}
\frac{1}{16nc}\left(1-\frac{1}{\sqrt{n}}\right)e^{\pi \sqrt{n}}<\N(a,c,n)<\frac{3}{16nc}\left(1+\frac{1}{\sqrt{n}}\right)e^{\pi \sqrt{n}},~~n\geq M_c,
\end{equation}
where $M_c$ is defined in \eqref{Mc-bound}.

Setting $n_2=Cn_1$ for some $C\geq1$, one can easily find that for $n_1\geq M_c$
\begin{align*}
  \N(a,c;n_1)\N(a,c;n_2)>\frac{1}{256Cn_1^2c^2}\left(1-\frac{1}{\sqrt{n_1}}\right)
  \left(1-\frac{1}{\sqrt{n_2}}\right)e^{\pi (\sqrt{n_1}+\sqrt{Cn_1})},
\end{align*}
and
\begin{align*}
  \N(a,c;n_1+n_2)<\frac{3}{16(n_1+Cn_1)c}\left(1+\frac{1}{\sqrt{n_1+Cn_1}}\right)e^{\pi \sqrt{n_1+Cn_1}}.
\end{align*}
It remains to show that for $n_1\geq M_c$
\begin{equation*}
  T_{n_1}(C)>\log \left(W_{n_1}(C)\right)+\log \left(S_{n_1}(C)\right),
\end{equation*}
where
\begin{align*}
W_{n_1}(C)&:=(48c)\frac {Cn_1}{C+1}.
\end{align*}
Obviously,
\begin{align*}
W_{n_1}(C)<48cn_1.
\end{align*}
Therefore, we only need show
\begin{align}\label{T-6}
   \notag
   T_{n_1}(1)&=2\pi \sqrt{n_1}-\pi \sqrt{2n_1}\\
   &>\log \left(48cn_1\right)+\log \left(\frac {1+\frac{1}{\sqrt{2n_1}}}{\left(1-\frac{1}{\sqrt{n_1}}\right)^2}\right).
\end{align}
If $n_1\geq2$, we obtain
\begin{align}\label{T-b1}
  \log \left(48cn_1\right)+\log \left(\frac {1+\frac{1}{\sqrt{2n_1}}}{\left(1-\frac{1}{\sqrt{n_1}}\right)^2}\right)
  <\log \left(48cn_1\right)+\log \left(\frac {1+\frac{1}{\sqrt{4}}}
  {\left(1-\frac{1}{\sqrt{2}}\right)^2}
  \right)<\log(840cn_1).
\end{align}
Moreover, if $n_1\geq(840c)^2\geq(840\times 6)^2$, then
\begin{equation}\label{T-b2}
   T_{n_1}(1)=2\pi \sqrt{n_1}-\pi \sqrt{2n_1}>2\log (n_1)\geq \log (n_1)+2\log (840c)>\log (840cn_1).
\end{equation}
Combined with \eqref{T-6}, \eqref{T-b1} and \eqref{T-b2}, choosing $n_1, n_2\geq(840c)^2$, we complete the proof.
\end{proof}

\vspace{0.5cm}
 \baselineskip 15pt
{\noindent\bf\large{\ Acknowledgements}} \vspace{7pt} \par
This work was supported by the National Natural Science Foundation of China (Grant Nos. 12001182 and 12171487),  the Fundamental Research Funds for the Central Universities (Grant No. 531118010411) and Hunan Provincial Natural Science Foundation of China (Grant No. 2021JJ40037).

\end{document}